\documentclass[12pt]{amsart}
\usepackage{amsmath,amsthm}
\input epsf

\newtheorem{theorem}{Theorem}
\newtheorem{lemma}{Lemma}

\def\IR {\text{\it IR}}
\def\a{{\buildrel {\text{ind}} \over \longrightarrow}}
\begin{document}
%--------------------------------------------------------------
\title[A NOTE ON LOWER BOUNDS FOR INDUCED RAMSEY NUMBERS]
{A NOTE ON  LOWER BOUNDS FOR INDUCED RAMSEY NUMBERS}
%\date {15 March 2000}
%--------------------------------------------------------------
\author{Izolda Gorgol}
\address{Department of Applied Mathematics\\
    Lublin University of Technology\\
    ul.\ Nadbystrzycka 38D, 20-618 Lublin\\
    Poland}
\email{i.gorgol@pollub.pl}
%--------------------------------------------------------------
%-----------------------------------------------------
\begin{abstract}
We say that a graph $F$ {\it strongly arrows} a pair of graphs $(G,H)$ and write $F \a (G,H)$ if  any 2-colouring of its edges with red and blue  
leads to  either a red $G$  or a blue $H$ appearing as induced subgraphs of $F$. 
{\it The induced Ramsey number}, $\IR(G,H)$ is defined as $\min\{|V(F)|: F\a (G,H)\}$.
We will consider two aspects of induced Ramsey numbers.
Firstly there will be shown that the lower bound of the induced Ramsey number for
a connected graph $G$ with independence number $\alpha$ and a  graph $H$
with clique number $\omega$  roughly $\frac{\omega^2\alpha}{2}$. This bounds is sharp.
Moreover we discuss also the case when $G$ is not connected providing also a sharp lower bound which is linear in both parameters.
\end{abstract}
%-----------------------------------------------------
\maketitle
%----------------------------------------------------

\section{INTRODUCTION}

We say that a graph $F$ {\it strongly arrows} a pair of graphs $(G,H)$ and write $F\a (G,H)$ if  any 2-colouring of its edges with red and blue  
leads to  either a red $G$  or a blue $H$ appearing as induced subgraphs of $F$. We call the graph $F$ as a strongly arrowing graph. 
{\it The induced Ramsey number}, $\IR(G,H)$ is defined as $\min\{|V(F)|: F\a (G,H)\}$. It is a generalization of  usual Ramsey numbers $R(G,H)$, where we color the edges of a complete graph and do not require the monochromatic copies to be induced. It is a corollary of the famous theorem of Ramsey that these numbers are always finite.

The existence of the induced Ramsey number is not obvious and it was a subject of an intensive study. Finally it was
 proved independently by
Deuber \cite{deuber},
Erd\H os, Hajnal and P\'osa \cite{erdhapo} and R\"odl \cite{rodl}. Since in case of complete graphs the induced subgraph is the same as the subgraph it is obvious that $\IR(K_m,K_n)=R(K_m,K_n)$. When at least one of the graphs in the pair is not complete these functions differ.

It is not much known about the behaviour of the induced Ramsey numbers.
The results are mostly of asymptotic type and concern upper bounds. It is surely motivated by the fact that these following from the above mentioned proofs are enormous and  Erd\H os conjectured \cite{erdirconj} that there is a positive constant $c$ such  that every graph $G$ with $n$ vertices satisfies $\IR(G,G)\le 2^{cn}$. The most recent result in that direction is that of Conlon, Fox and Sudakov \cite{confoxsud} who showed that $\IR(G,G)\le 2^{cn\log n}$ improving the erliear result of  Kohayakawa, Pr\"omel and R\"odl  $\IR(G,G)\le 2^{cn(\log n)^2}$ \cite{kohprro}.

% \cite{becksize1, haxkohlu,  luczdeg}.

Moreover these results are generally upper bounds obtained either by probabilistic (\cite{becksize1, haxkohlu, kohprro, luczdeg}) or by constructive methods.  A comparision of results of both types can be found in the paper of Shaefer and Shah \cite{schaefer}. The authors give there the construction of arrowing graphs for a number of pairs of graphs including trees, complete graphs, bipartite graphs and cycles.

%Similarly as for lower bounds some general bound is known for triangle-free graphs versus complete graphs. It was established by Gorgol \cite{gorgol} and is actually true for connected triangle-free graphs.

As for  the lower bound, it is obvious by the definition 

\begin{equation}
\IR(G,H) \ge R(G,H)
\end{equation}
 and, as far as we know,  it is the only general lower bound known so far.

The main result of this short note is Theorem \ref{dol} which establishes the lower bound for the induced Ramsey number in terms of independence and clique numbers. Although the inductive proof of this theorem is not complicated the result is the first step in that direction. The theorem is somehow similar in spirit to the result of Chv\'atal and Harary \cite{chvatalharary} for Ramsey numbers who observed that for connected $G$

\begin{equation}\label{chi}
R(G,H)\ge (|V(G)|-1)(\chi(H)-1)+1
\end{equation}

To prove \eqref{chi}, consider a 2-edge-coloring of the complete graph on $(|V(G)|-1)(\chi(H)-1)$ vertices consisting of $(\chi(H)-1)$ disjoint red cliques of size $|V(G)|-1$ . This coloring has no red $G$ because all red connected components have size $|V(G)|-1$, and there is no blue $H$ since the partition of this $H$ induced by red cliques would give a coloring of $H$ by $\chi(H)-1$ colors.

For some graphs the bound in \eqref{chi} is quite far from the truth. For example Erd\H os \cite{erdramskn} showed that
which is much larger than the quadratic bound we get from \eqref{chi}. 

%Burr \cite{burr}

This bound is sharp in that sense that for a pair: a star versus a complete graph this lower bound is actually the exact value of the induced Ramsey number. This theorem has its nontrivial application in all cases that the independence (clique) number differs not much from the number of vertices of the graph.  

At the end we mention that 
the only known exact values (not concerning the pairs of small graphs) are for a pair of stars by Harary, Ne\v set\v ril and R\"odl \cite{irn}, matching versus complete graphs by Gorgol and {\L}uczak \cite{gorgolluczak} and for stars versus complete graphs by Gorgol \cite{gorgol}. The two latter will serve as examples of sharpness of our theorems.

\section{Notation}

In this paper we do not introduce any special notation. A graph $G$ is a subgraph of a graph $H$ (denoted by $G\subset H$) if $V(G)\subset V(H)$ and $E(G)\subset E(H)$. A graph $G$ is an induced subgraph of a graph $H$ (denoted by $G\prec H$) if $V(G)\subset V(H)$ and $E(G)=\{uv\in E(H): u,v\in V(G)\}$.
By $F[S]$ we mean a graph induced by a vertex-set $S$. Let $t$ be a positive integer and $F$ be a graph. By a symbol $tF$ we mean a graph consisting of $t$ disjoint copies of the graph $F$. 
For graphs $G$, $H$ 
the symbol $G\cup H$ denotes  a disjoint sum of graphs 
 and $G\setminus H$ denotes a graph obtained from 
$G$ by removing a subgraph $H$ (with all incident edges). The  independence number of a graph $G$, i.e. the size of the largest set of mutually nonadjacent vertices, we denote by  $\alpha(G)$, the clique number, i.e. the size of the largest clique, by $\omega(G)$ and the chromatic number, i.e.  the smallest number of colors needed to color the vertices of $G$ so that no pair of adjacent vertices have the same color by $\chi(G)$. The symbols $P_n$, $C_n$, $K_n$ stand for a path, a cycle and a complete graph on $n$ vertices respectively and $S_k$ for a star with $k$ rays.

\section{Lower bounds for an induced Ramsey number}

As we mentioned in the introduction it is very little known about lower bounds for the induced Ramsey numbers  and a natural lower bound is  the usual Ramsey number. To prove the lower bound for the induced Ramsey number we should show that we can colour every graph $F$ with prescribed number of vertices without induced monochromatic copies of given graphs $(G,H)$. It forces to examine not only the number of vertices of the graph $F$ but also deeply its structure. It occurs that if we consider the independence number of the graph $G$ and the clique number of the graph $H$, as somehow opposite notions, it is not so difficult to  deduce something about the structure of an arrowing graph $F$. Then we can construct an appropriate colorings. The constructions described below  mainly arose from the fact that since $\omega(H)=\omega$ than $K_\omega$ is contained in $F$ and a subgraph and an induced subgraph in case of cliques is the same. Therefore if we avoid a blue $K_\omega$ we avoid an induced $H$ as well. The construcion of colourings avoids connected red graphs with the independence number $\alpha$ simply by taking red subgraph in which the independance number is maximally $\alpha -1$.

\begin{theorem}\label{dol}
Let $G$ be an arbitrary connected graph with $\alpha(G)=\alpha\ge 2$  and $H$ be an arbitrary graph with $\omega(H)= \omega$. Then
$$\IR(G,H)\ \ge \ (\alpha-1)\frac{\omega(\omega-1)}2+\omega.$$
\end{theorem}
\begin{proof}

Let $F$ be an arbitrary graph on $(\alpha-1)\frac{\omega(\omega-1)}2+\omega - 1$  vertices. We shall show that
$F$ can be 2-coloured with no red induced $G$ and no blue induced $H$.

The proof will be conducted by induction on $\omega$. It is trivial
for $\omega=2$.
Note that certainly $F$ contains a clique $K_\omega$ otherwise it could be coloured blue.
Let us denote this clique $K^0$ and colour it red. It is easy to observe
that $F\setminus K^0$ must contain a clique $K_{\omega-1}$ otherwise we could colour
the remaining edges of $F$ blue. Denote this clique $K^1$ and colour $F_1=F[V(K^0)\cup V(K^1)]$
 red. Similarly $F\setminus F_1$
must contain a clique $K_{\omega-1}$ which we denote by $K^2$ .
Repeating the above consideration we conclude
that apart from $K^0$ the graph $F$ contains $\alpha-2$ disjoint cliques $K_{\omega-1}$ denoted
by $K^1$, $K^2$, $\dots$, $K^{\alpha-2}$.
Let  all edges of $F[V(\bigcup_{i=0}^{\alpha-2} K^i)]$ be red.
Let $F'=F\setminus\bigcup_{i=0}^{\alpha-2} K^i$. Since $F'$ fulfils the inductive assumption,
 it can be 2-coloured with no red induced $G$ and no blue induced $H$.
Let all not coloured so far edges of $F$ be blue.
In such a colouring there is no red induced $G$ and no blue induced
$H$ and so $\IR(G,H) > (\alpha-1)\frac{\omega(\omega-1)}2+\omega - 1.$

\end{proof}

It follows from the above inductive proof that the strongly arrowing graph $F$ must contain a number of disjoint cliques. Precisely $\bigcup_{j=0}^{\omega-2}(K_{\omega-j}\cup (\alpha-2)K_{\omega-j-1})$ is a subgraph of $F$.

As we mentioned the lower bound from Theorem \ref{dol} is sharp. Gorgol \cite{gorgol} showed the exact value of the induced Ramsey number for stars versus complete graphs.

\begin{theorem}\cite{gorgol}
For arbitrary $k\ge 1$ and $n\ge 2$ holds
$$\IR(S_k,K_n)\ =\ (k-1)\frac{n(n-1)}2+n.$$
\end{theorem}

On the other hand if we take a path instead of the star we obtain $\IR(P_t,K_n)\ \geq\ (\lceil \frac t2\rceil-1)\frac{n(n-1)}2+n$. Comparing this with the well known result of  Chv\`atal \cite{chvatal} $R(T,K_n)=(t-1)(n-1)+1 $, where $T$ denotes a tree on $t$ vertices, we obtain a better result for big complete graphs obtaining quadratic instead of a linear bound. The similar can be said if we take cycle instead of path, assuming that the conjecture of Erd\H os ($R(C_t,K_n)=(t-1)(n-1)+1 $) is true. 

Moreover Kohayakawa, Pr\"omel and R\"odl \cite{kohprro} showed that the induced Ramsey number of a tree $T$ and any graph $H$ grows 
polynomially with $|T|=t$ and $|H|=n$

$$ \IR(T,H)\le ct^2n^4\Bigl(\frac{\log (tn^2)}{\log\log\log(tn^2)}\Bigr).$$

Some more effort is needed to proof the analogous lower bound if we allow the graph $G$ not necesserily to be connected. However we assume that it does not contain isolates.
The proof proof is inductive again, but now the first step requires a little bit more attention. 

\begin{lemma} \label{alfa2}
Let $G$ be an arbitrary isolates-free graph with $\alpha(G)= 2$  and $H$ be an arbitrary graph with $\omega(H)\ = \omega$. Then
$$\IR(G,H)\ \ge\ 2\omega.$$
\end{lemma}

\begin{proof}
Let $F$ be an arbitrary graph on $2n-1$ vertices. We wil show that it can be coloured without red induced graph with $\alpha = 2$ and blue indyced graph with clique number $\omega$. Obviously $F$ is not complete. Assume that $x$, $y$ are such that $xy\not\in E(F)$. In construction of our coloring we take into account the following fact:

%(i) if $z\in N(x)\cap N(y)$ then at least one of the edges $xz$, $yz$ must be blue

(i) if there exist two independent red edges $xu$ and $yv$, then at least one the edges $uv$, $xv$, $yu$ must exist and be blue.

Note that if $K_t\subset F$ and $t>\omega$ then we color this $K_t$ red, all remaining edges blue and we are done. Hence we assume that there is no $K_{\omega + 1}$ in $F$. On the other hand $F$ must contain $K^1=K_\omega$ otherwise could be coloured blue. Similary like in the proof of Theorem \ref{dol} coloring this $K_\omega$ red we conclude that the remaining $\omega-1$ vertices form a clique $K^2=K_{\omega-1}$.  

If $K^2$ do not form $K_\omega$ with some vertices from $K^1$ then we color $K^1$ red
 and all remaining edges blue.
 
 Thefore assume that there exists $A\subset V(K^1)$ and $C\subset V(K^2)$ such that $F[A\cup B]=K_\omega=K^3$. Let $s=|A|$. Certainly $1\leq s\leq n-1$. Let us take $A$ with maximum  $s$. Let $B=V(K^1)\setminus A$ and $D=V(K^2)\setminus C$.
 
If $s=1$ then we have two cliques $K_\omega$ sharing one vertex, say $a$. Then we color with red two edges: $ab$ for arbitrary $b\in B$ and $c_1c_2$ for $c_1$, $c_2\in C$ and the remaining edges blue. This coloring fulfils (i).

Let now $s\geq 2$. Note that apart from $K^1$ and $K^3$ the graph $F$ may contain at most two more cliques $K_\omega$. There exists at most one $a\in A$ such that $F[B\cup D\cup \{a\}]=K_\omega$ and  at most one $c\in C$ such that $F[B\cup D\cup \{c\}]=K_\omega$. For more then one such vertices we would obtain a larger clique.

We color with red the edges $aa_1$, $ab$, $cd$, where $a_1\in A$, $b\in B$, $d\in D$ are chosen arbitrarily. If any of these additional cliques does not exists, adequate $a$ and $c$ we can also choose arbitrarily. This coloring also fulfils (i).
\end{proof}

\begin{theorem} \label{alfaomega}
Let $G$ be an arbitrary isolates-free graph with $\alpha(G)=\alpha\ge 2$  and $H$ be an arbitrary graph with $\omega(H)\ = \omega$. Then
$$\IR(G,H)\ \ge\ \alpha\omega.$$
\end{theorem}

\begin{proof}

Let $F$ be  an arbitrary graph
on~$\alpha\omega-1$ vertices.
We use the induction on $\alpha$ to prove 
 that $F$ can be 2-coloured with no red induced $G$
and no blue induced $H$.

The assertion   for $\alpha=2$ follows from Lemma \ref{alfa2}.
Thus, let $\alpha>2$. 
We may assume that $G$ contains a clique $K_\omega$;   
otherwise we  colour all edges of $F$  blue.
Colour this clique red.
A graph induced by the remaining vertices fulfills the 
inductive assumption so it can be coloured with no red induced graph with independence number
$(\alpha-1)$ and no blue induced $H$.
Now,  colour red all edges of~$F$ which have not been 
 coloured so far.
\end{proof}

It is worth to notice that if we allow the graph $G$ not to be connected this lower bound is sharp. Gorgol and {\L}uczak \cite{gorgolluczak} showed the exact value of the induced Ramsey number for a matching and a complete graph.

\begin{theorem} \cite{gorgolluczak} \label{skojind}
For arbitrary  $k\ge 1$ and~$n\ge 2$
$$\IR(kK_2,K_n)\ =\ kn.$$
\end{theorem}

\bibliographystyle{amsplain}

\begin{thebibliography}{1}

\bibitem{becksize1}
J.~Beck, \emph{On the size {R}amsey number of paths, trees and circuits {I}},
  J. Graph Theory \textbf{7} (1983), 115--129.
  
\bibitem{burr}
S. Burr,
\emph{Ramsey numbers involving graphs with long suspended paths}
J. Lond. Math. Soc., \textbf{24} (1981), pp. 405--413.

\bibitem{confoxsud}
 D. Conlon, J. Fox and B. Sudakov, \emph{On two problems in graph
Ramsey theory}, Combinatorica \textbf{32}(2012), 513–-535.
  
\bibitem{chvatalharary}
 V.  Chv\'atal  and  F.  Harary,  \emph{Generalized  Ramsey  theory  for  graphs.  III.  Small  off-diagonal
numbers}, Pacific J. Math. \textbf{41} (1972), 335–-345.
  
\bibitem{chvatal}
 {V. Chv\'atal}, \emph{Tree-complete graph {R}amsey number}, J. Graph Theory \textbf{1} (1977), 93.
 
  

\bibitem{deuber}
W.~Deuber, \emph{A generalization of {R}amsey's theorem}, Infinite and finite
  sets (R.~Rado A.~Hajnal and V.~S\'os, eds.), vol.~10, North-Holland, 1975,
  pp.~323--332.

\bibitem{erdramskn}
P. Erd\H os,
\emph{Some remarks on the theory of graphs},
Bull. Amer. Math. Soc., 53 (1947), pp. 292--294

\bibitem{erdirconj}
P. Erd\H os, \emph{On some problems in graph theory, combinatorial analysis and combinatorial
number theory.} Graph theory and combinatorics (Cambridge, 1983) (1984),
1–-17, Academic Press, London, 1984.

\bibitem{erdhapo}
P.~{Erd\H os}, A.~Hajnal, and L.~P\'osa, \emph{Strong embeddings of graphs into
  colored graphs}, Infinite and finite sets (R.~Rado A.~Hajnal and V.S\'os,
  eds.), vol.~10, North-Holland, 1975, pp.~585--595.

\bibitem{gorgol}
I. Gorgol, \emph{A note on a triangle-free–complete graph induced Ramsey number},
Discrete Math. \textbf {235}, 1--3 (2001), 159–-163.

\bibitem{gorgolluczak}
I. Gorgol, T. {\L}uczak, \emph{On induced Ramsey numbers}, Discrete Math. \textbf{251},
1--3 (2002), 87–-96.

\bibitem{irn}
F. Harary, J. Ne\v set\v ril  and V. R\"odl, \emph{Generalized Ramsey theory for graphs.
XIV. Induced Ramsey numbers}, Graphs and other combinatorial topics (Prague,
1982) 59 (1983), 90–-100.


\bibitem{haxkohlu}
P.~Haxell, Y.~Kohayakawa, and T.~{\L}uczak, \emph{The induced size-{R}amsey
  number of cycles}, Combinatorics, Probab. Comput. \textbf{4} (1995),
  217--240.

\bibitem{kohprro}
Y.~Kohayakawa, H.J. Pr\"omel, and V.~R\"odl, \emph{Induced {R}amsey numbers},
  Combinatorica \textbf{18} (1998), no.~3, 373--404.
  
\bibitem{kostoczka}
A. Kostochka and N. Sheikh, \emph{On the induced Ramsey number IR(P3;H)},
Topics in discrete mathematics 26 (2006), 155–-167.

\bibitem{luczdeg}
T.~{\L}uczak and V.~R\"odl, \emph{On induced {R}amsey numbers for graphs with
  bounded maximum degree}, J. Combin. Theory, Ser. {\bf B} \textbf{66} (1996),
  324--333.

\bibitem{rodl}
V.~R\"odl, \emph{A generalization of {R}amsey theorem}, Ph.D. thesis, Charles
  University, Prague, Czech Republic, 1973, pp.~211--220.

\bibitem{schaefer}
M. Schaefer and P. Shah, \emph{Induced graph Ramsey theory}, Ars Combin. \textbf{66}
(2003), 3-–21.
\end{thebibliography}

%\newpage

%\newpage

\end{document}